\documentclass[a4paper]{article}

\usepackage[T1]{fontenc}
\usepackage[latin1]{inputenc}
\usepackage[english]{babel}
\usepackage{cite}

\usepackage{amsmath}
\usepackage{amssymb}
\usepackage{amsthm}
\usepackage{mathtools}
\usepackage{dsfont}

\usepackage{tikz}
\usepackage{graphicx}
\usetikzlibrary{cd,arrows,chains,matrix,positioning,scopes}
\usepackage{adjustbox}

\usepackage{hyperref}
\hypersetup{colorlinks,%
citecolor=black,%
filecolor=black,%
linkcolor=black,%
urlcolor=black}

\numberwithin{equation}{section}

\theoremstyle{remark}
\newtheorem{remark}{Remark}[section]
\newtheorem{example}[remark]{Example}
\theoremstyle{definition}

\theoremstyle{plain}
\newtheorem{theorem}[remark]{Theorem}
\newtheorem{lemma}[remark]{Lemma}

\newtheorem{proposition}[remark]{Proposition}

\renewcommand{\O}{\mathcal{O}}

\DeclareMathOperator{\jac}{Jac}
\DeclareMathOperator{\Alb}{Alb}
\DeclareMathOperator{\alb}{alb}
\DeclareMathOperator{\Hom}{Hom}
\DeclareMathOperator{\Pic}{Pic}
\DeclareMathOperator{\ns}{NS}

\DeclareMathOperator{\im}{Im}
\DeclareMathOperator{\aut}{Aut}

\DeclarePairedDelimiter{\intinf}{\lfloor}{\rfloor}

\newcommand\restr[2]{{
  \left.\kern-\nulldelimiterspace 
  #1 
  \vphantom{\big|} 
  \right|_{#2} 
  }}

\title{Surfaces close to the Severi lines}
\author{Federico Conti}

\begin{document}

\maketitle

\begin{abstract}
 Let $X$ be a surface of general type with maximal Albanese dimension: if $K_X^2<\frac{9}{2}\chi(\O_X)$, one has $K_X^2\geq 4\chi(\O_X)+4(q-2)$. We give a complete classification of surfaces for which equality holds for $q(X)\geq 3$: these are surfaces whose canonical model is a double cover of a product elliptic surface branched over an ample divisor with at most negligible singularities which intersects the elliptic fibre twice.

We also prove, in the same hypothesis, that a surface $X$ with $K_X^2\neq 4\chi(\O_X)+4(q-2)$ satisfies $K_X^2\geq 4\chi(\O_X)+8(q-2)$ and we give a characterization of surfaces for which the equality holds. These are surfaces whose canonical model  is a double cover of an isotrivial smooth elliptic surface branched over an ample divisor with at most negligible singularities whose intersection with the elliptic fibre is $4$.
\end{abstract}

\tableofcontents

\section{Introduction}

Let $X$ be a minimal surface of general type of maximal Albanese dimension (recall that a surface is called of maximal Albanese dimension if its Albanese morphism is generically finite). We denote by $K_X$ the canonical divisor, by $\chi(\O_X)$ the Euler characteristic of the structure sheaf and by $q=h^1(\O_X)$ the irregularity. 

In this paper we are interested in characterizing  surfaces which lie on or close to the Severi lines, i.e. surfaces for which the quantity
\begin{equation}
\label{int}
K_X^2-4\chi(\O_X)-4(q-2)
\end{equation}
vanishes or is "small" provided that $K_X^2<\frac{9}{2}\chi(\O_X)$. This value is strictly related to the so called Severi inequality (cf. \cite{par_sev}), which states that a surface of general type of maximal Albanese dimension satisfies
\begin{equation}
K_X^2\geq 4\chi (\O_X).
\label{sev}
\end{equation}
In \cite{barparsto} there is a characterization of surfaces for which the inequality \ref{sev} is indeed an equality, namely these are surfaces whose canonical model is a double cover of its Albanese variety branched over an ample divisor with at most negligible singularities (in particular $q=2$). There are many generalizations of the Severi inequality; in particular Lu and Zuo have proved in \cite{lu} a similar inequality involving also the irregularity $q$:  a surface of general type and maximal Albanese dimension satisfies
\begin{equation}
K_X^2\geq\min\Bigl\{\frac{9}{2}\chi(\O_X),4\chi(\O_X)+4(q-2)\Bigr\}
\label{lusev}
\end{equation}
or, equivalently, if $K_X^2<\frac{9}{2}\chi(\O_X)$ then $K_X^2\geq 4\chi(\O_X)+4(q-2)$.
They also give conditions for a surface to satisfy the equality
\begin{equation}
\frac{9}{2}\chi(\O_X)>K_X^2=4\chi(\O_X)+4(q-2).
\label{uglu}
\end{equation}
The condition $K_X^2<\frac{9}{2}\chi(\O_X)$ is necessary to prove that there exists an involution $i$ for which the Albanese morphism of $X$ is composed with $i$ (cf. \cite{lu} Theorem 3.1) which is central in their argument. 
There is a single step, \cite{lu} Lemma 4.4(2), where the condition $K_X^2<\frac{9}{2}\chi(\O_X)$ is really needed in their proof and it is not enough to require that  the Albanese morphism of $X$ is composed with an involution.

The first main result of this paper is a complete characterization of surfaces satisfying $K_X^2=4\chi(\O_X)+4(q-2)$ in case $q\geq 3$ and $K_X^2<\frac{9}{2}\chi(\O_X)$: Lu and Zuo have proved that the canonical model of such a surface is a double cover of a smooth isotrivial elliptic surface branched over a divisor $R$ with at most negligible singularities. We prove here that this elliptic surface has to be a product $C\times E$ and we also determine the linear class of the branch divisor. 

\begin{theorem}
\label{mioa}
Let $X$ be a surface of general type with maximal Albanese dimension satisfying $q=q(X)\geq 3$ such that $K_X^2<\frac{9}{2}\chi(\O_X)$.
Then 
\begin{equation*}
K_X^2= 4\chi(\O_X) +4(q-2)
\end{equation*}
if and only if the canonical model of $X$ is isomorphic to a double cover of a product elliptic surface $Y=C\times E$ where $E$ is an elliptic curve and  $C$ is a curve of  genus $q-1$, whose branch divisor $R$ has at most negligible singularities and 
\[R\sim_{lin}C_1+C_2+\sum_{i=1}^{2d}E_i,\]
where the $E_i$ (respectively $C_i$) are fibres of the first projection (respectively the second projection) of $C\times E$ and $d>7(q-2)$. Moreover, we have that $\Alb(X)\simeq\Alb(Y)$.
\end{theorem} 

In Example \ref{es1}, we will give a relation between the invariants of $X$ and the number $d$ appearing in the linear class of the branch divisor $R$. Actually we will see that $K_X^2=8(q-2)+4d$ and $\chi(\O_X)=q-2+d$ and we will give a construction of a surface satisfying the hypotheses of Theorem \ref{mioa} for every $d>7(q-2)$ (this inequality is required to satisfy the hypothesis $K_X^2<\frac{9}{2}\chi(\O_X)$). In particular, for every $q\geq 3$, this gives an unlimited set of couples $(a,b)$, for which there exists such a surface with invariants $K_X^2=a$ and $\chi(\O_X)=b$.

The second result is about surfaces that are not on the Severi lines but are close to them. We see that in this case $K_X^2\geq 4\chi(\O_X)+ 8(q-2)$ and we also give a characterization for surfaces that satisfy this equality.

\begin{theorem}
\label{miob}
Let $X$ be a minimal surface of general type with maximal Albanese dimension with $K_X^2<\frac{9}{2}\chi(\O_X)$.
\begin{enumerate}
\item If $K^2_X>4\chi(\O_X)+4(q-2)$, then $K^2_X\geq 4\chi(\O_X)+8(q-2).$
\item If $q=2$ and $K^2_X>4\chi(\O_X)$, then $K^2_X\geq 4\chi(\O_X)+2.$
\end{enumerate}
Moreover, if $q\geq 3$, equality holds, i.e.
\[K^2_X=4\chi(\O_X)+8(q-2),\]
if and only if the canonical model of $X$ is isomorphic to a double cover of a smooth isotrivial elliptic surface $Y$ over a curve $C$ of genus $q-1$, branched over a divisor $R$ with at worst negligible singularities for which $K_Y.R=8(q-2)$. In particular, we have that $\Alb(X)\simeq\Alb(Y)$.
\end{theorem}

We would like to stress that all the inequalities in Theorems \ref{mioa} and \ref{miob} are sharp for every $q$: in section 3 we give examples for which the equalities hold.

\paragraph{Notation and conventions}
We work over the complex numbers. All varieties are supposed to be projective. Given a surface $S$ we denote by $\Alb(S)$ its Albanese variety and by $\alb_S\colon S\to \Alb(S)$ its Albanese morphism. In this paper $X$ is a surface of general type with maximal Albanese dimension, $C$ is a curve of genus $g>1$, $E$ is an elliptic curve. 

Given the product $C\times E$, we denote by $\pi_C$ and $\pi_E$ the two projections respectively to $C$ and $E$, and by $E_c$ the fibre of $\pi_C$ over $c\in C$ (sometimes $E$ or $E_i$ if it is not necessary to specify the point $c$), respectively $C_e$ the fibre of $\pi_E$ over $e\in E$. Given $L_C\in\Pic(C)$ and $L_E\in\Pic(E)$ we denote by $L_E\boxtimes L_C=\pi_C^*(L_C)\otimes\pi_E^*(L_E)$. By $c_0$ we mean a fixed point of $C$ and we denote by $\Hom_{c_0}(C,E)$ the group of homomorphisms between $C$ and $E$ which send $c_0$ to the origin of $E$ (the group structure is given by the one on $E$). For every $f\in\Hom(C,E)$ we denote by $f+e\in \Hom(C,E)$ the morphism given by $c\mapsto f(c)+e$ and by $\Gamma_f$ its graph as a divisor on $C\times E$.

We use interchangeably the notion of line bundles and Cartier divisors and we use both additive and multiplicative notations.

\paragraph{Acknowledgement} The author would like to thank his advisor Rita Pardini for useful mathematical discussion concerning the topics of the paper. The author is also grateful to Davide Lombardo for his advice on the Picard group of a product of curves.

\section{Preliminaries}
In this section we describe the constructions and we expose preliminary results which will be needed in the proofs of Theorems \ref{mioa} and \ref{miob}.

\subsection{Picard group of $C\times E$}
\label{picce}
The Picard group of the product of two curves is known. Here we recall the formula for $\Pic(C\times E)$ where $E$ is an elliptic curve and $C$ is a curve of genus $g>1$ and we also see how it behaves in the equivariant setting. The last Lemma of this section gives a sufficient condition for the quotient $(C\times E)/G$ (where $G$ is a group acting freely on both $C$ and $E$) to be isomorphic to $C/G\times E$ in terms of the Picard group of $C\times E$.

Denote by $\Hom_{c_0}(C,E)$ the group of morphisms between $C$ and $E$ for which the image of $c_0\in C$ is the origin of the elliptic curve (the group structure is given by the one of $E$). Denote by $i_c\colon E\to C\times E$ the inclusion defined by $e\mapsto (c,e)$.

\begin{proposition}
\label{Teo}
In the above settings we have the following split exact sequence of groups:
\begin{equation}
0\rightarrow \Pic (C)\times \Pic(E) \xrightarrow{\alpha} \Pic(C\times E) \xrightarrow{\beta} \Hom_{c_0}(C,E) \rightarrow 0,
\label{pics}
\end{equation}
where $\beta$ is defined by $\beta(D)(c)=i_c^*(D)-i_{c_0}^*(D)$ (here we are using the isomorphism $E\simeq \jac(E)$ given by the Abel-Jacobi map) and the section $s$ of $\beta$ is given by
\[s\colon \Hom_{c_0}(C,E) \to  \Pic(C\times E)\quad s(f)=\Gamma_f-C_0-\sum_{c\in f^{-1}(0)} a_cE_c,\]
where $a_c$ is  the multiplicity of $f$ at $c$.
\end{proposition}

For the proof of this Proposition we refer to \cite{lange} proposition 11.5.1. Notice that in our statement we are using the canonical isomorphism, which we call $\phi$, 
\[\Hom_{c_0}(C,E)\simeq\Hom(\jac(C),\jac(E)),\]
which is induced  by the Abel-Jacobi map $AJ_C\colon C\to\jac(C)$ and the canonical isomorphism of $E$ with its Jacobian variety:
\[
\begin{tikzcd}
C\arrow{d}{AJ_C} \arrow{dr}{f} &  \\
\jac(C) \arrow[swap]{r}{\phi(f)} & E\simeq\jac(E).
\end{tikzcd}
\] 

By this, it is possible to define 
\[\beta\colon\Pic(C\times E)\to \Hom_{c_0}(C,E)\]
as
\[\O(D)\mapsto \phi^{-1}(f_D)\]
where $f_D$ is a morphism in $\Hom(\jac(C),E)$ defined as 
\[f_D(P)=(\pi_E)_*(\pi^*_CP\cdot D):\]
we refer to \cite{fulton} for pushforward and pullback of cycles of a proper flat morphism (chapter 1), for intersection product of cycles of a smooth variety (chapter 8) and their behaviour with respect to linear equivalence. It is immediate by the definition that $f_D$ is a morphism of groups.

 Let $i_c\colon E\to C\times E$ be as in the statement, we see that $i_c^*(D)=(\pi_E)_*(D\cdot E_c)$ where $E_c$ is the fibre of $\pi_C$ over $c\in C$. By this 
\[\beta(D)(c)=\phi^{-1}(f_D(c-c_0))=\phi^{-1}((\pi_E)_*((E_c-E_{c_0})\cdot D))=i_c^*(D)-i_{c_0}^*(D).\]

Notice also that for $f\colon C\to E$ a morphism which sends $c_0$ to $0$ and let $\Gamma_{f+e}$ be the graph of the morphism $(f+e)(c)=f(c)+e$ inside the product $C\times E$, where $e$ is a point of $E$. We see that 
\begin{equation}
\beta(\Gamma_{f+e})(c)=i_c^*(\Gamma_{f+e})-i_{c_0}^*(\Gamma_{f+e})=(f(c)+e)-(f(c_0)+e)=f(c).
\label{inve}
\end{equation}
By a similar argument, $\beta$ is invariant under translation by $e\in E$ for a general divisor $D\in\Pic(C\times E)$.

We recall here the See-saw Principle (cf. \cite{mumford} Corollary II.5.6), which can be used to prove the exactness of the sequence in Proposition \ref{Teo} and which we will need later in the proof of Theorem \ref{mioa}.

\begin{theorem}[See-saw Principle]
\label{saw}
Let $A$ and $B$ be two smooth curves, $\pi_2\colon A\times B\to B$ be the second projection and let $L\in\Pic(A\times B)$ be a line bundle such that $\restr{L}{A\times\{b\}}$ is trivial for every $b\in B$. Then there exists $L_B\in\Pic(B)$ such that $L=\pi_2^*L_B$. 
\end{theorem} 

Now we give an equivariant version of Proposition \ref{Teo}.

\begin{proposition}
\label{rmpic}
Suppose there exists a finite Abelian group $G$ acting freely on $C$, $E$ and diagonally on $C\times E$ (i.e. $g\cdot(c,e)=(g\cdot c,g\cdot e)$). Then it is possible to give to $\Pic(C)$, $\Pic(E)$, $\Pic(C\times E)$ and $\Hom_{c_0}(C,E)$ a $G$-module structure such that
\begin{equation}
0\rightarrow \Pic (C)\times \Pic(E) \xrightarrow{\alpha} \Pic(C\times E) \xrightarrow{\beta} \Hom_{c_0}(C,E) \rightarrow 0,
\end{equation}
is an exact sequence of $G$-modules, where $\alpha$ and $\beta$ are the same morphisms defined in Proposition \ref{Teo}.
\end{proposition}
\begin{proof}
Recall that, if $X$ is a variety and $G$ a group acting on $X$, we can naturally define an action of $G$ on $\Pic(X)$ given by $g\cdot L=g_*L$ where, with an abuse of notation, we are identifying $g$ with the corresponding automorphism of $X$; equivalently we easily see that this action is induced by $g\cdot D=g_*D=\{x\in X\ |\ \exists\ d\in D\ \text{with}\ g\cdot d=x\}$ at the level of divisors. Hence, it is clear that the morphism $\alpha$ in the statement is a morphism of $G$-modules. Moreover, we can consider $G$ as a finite subgroup of $E$ when considering its action on $E$ (hence we use the additive notation for this factor), while we use the multiplicative notation when considering $G$ acting on $C$. Clearly, the action of $G$ on the Picard groups of $C$, $E$ and $C\times E$ is faithful (because two points on a curve of genus greater or equal than $1$ are never linearly equivalent) but not free (every divisor of the form $\sum_{g\in G}g\cdot D$ is fixed by every $g\in G$).

Now, we would like to give to $\Hom_{c_0}(C,E)$ the structure of a $G$-module such that the morphism $\beta$ preserves the $G$-module structure. In order to do this we see how $G$ acts on divisors $\Gamma_f$ which are graphs  of functions $f$ in $\Hom_{c_0}(C,E)$. We see that
\[g\cdot \Gamma_f=\{(g c,g+f(c))\ | \ c\in C\}=\{(c,f(g^{-1} c)+g)\ | \ c\in C\}=\Gamma_{f_g},\]
where $f_g$ is defined by $c\mapsto f(g^{-1}c)+g$. However, in general $f_g\notin \Hom_{c_0}(C,E)$, but if we take $g\cdot f(c)=f_g(c)-f_g(c_0)=f(g^{-1}c)-f(g^{-1}c_0)$, then we see that $g\cdot f\in\Hom_{c_0}(C,E)$. Then we see that $f\mapsto g\cdot f$ is a well defined action of $G$ on $\Hom_{c_0}(C,E)$: indeed the axioms are easily verified. We have already noticed that $\beta$ is invariant under translation by $e\in E$ (Equation \ref{inve}); moreover, by the splitting exact sequence \ref{pics}, we know that every divisor $D$ on $C\times E$ is linearly equivalent to
\[\sum_{i}E_{c_i}+\sum_{j}C_{e_j}+\Gamma_{f}\]
with $f=\beta(\O(D))$ and suitable $c_i$ and $e_j$. These two facts implies that $\beta$ preserves the $G$-module structure.
\end{proof}

\begin{remark}
Notice that if $\Hom_{c_0}(C,E)$ is not trivial, then there are a lot of morphisms fixed by $G$. Actually, if $f$ is a nontrivial morphism in $\Hom_{c_0}(C,E)$ and $n$ is the order of $G$, then $nf$ is a nontrivial element fixed by every $g\in G$. Indeed (recall that $f_g$ denotes the morphism from $C$ to $E$ for which $g\cdot\Gamma_f=\Gamma_{f_g}$) the divisor 
\[\sum_{g\in G}g\cdot\Gamma_f=\sum_{g\in G}\Gamma_{f_g}\]
is fixed by the action of $G$ and its image via $\beta$ is $nf$. Hence $nf$ has to be fixed by the action of $G$ too.
\end{remark}

The following Lemma gives a sufficient condition for the quotient of a product elliptic surface to be trivial and will be fundamental in the proof of Theorem \ref{mioa}.

\begin{lemma}
\label{lem2}
In the same settings as Proposition \ref{rmpic}, the elliptic fibration $(C\times E)/G\to C/G$  with general fibre $E$ is trivial if and only if there is a line bundle $L$ on $C\times E$ which is fixed by the action of $G$ for which $L.E=1$ where $E$ (by an abuse of notation) is a general fibre of the first projection.
\end{lemma}

\begin{proof}
Denote by $\pi_{C/G}\colon(C\times E)/G\to C/G$, $\pi_C\colon C\times E\to C$ the two elliptic fibrations and by $\phi\colon C\times E\to (C\times E)/G$ and $\psi\colon C\to C/G$ the two quotients by $G$. 
Suppose that  $(C\times E)/G\to C/G$ is a product elliptic fibration: in particular there exists a section $i_{C/G}\colon C/G\to (C\times E)/G$ of $\pi_{C/G}$.
Then, the pull-back $i_C$ of $i_{C/G}$ to $C\times E$ is a section of $\pi_C$, i.e.
\begin{equation}
\begin{tikzcd}
C\arrow[dr, phantom, "\square"]\arrow[bend right=80, swap]{dd}{Id}\arrow{r}{\psi}\arrow{d}{i_C} & C/G\arrow{d}{i_{C/G}}\arrow[bend left=80]{dd}{Id}\\
C\times E \arrow[dr, phantom, "\square"] \arrow{r}{\phi} \arrow{d}{\pi_C} & (C\times E)/G \arrow{d}{\pi_{C/G}}\\
C\arrow{r}{\psi} & C/G
\end{tikzcd}
\label{eq_variabella}
\end{equation}
This means that  $\phi^*\O_{(C\times E)/G}{(\im(i_{C/G}))}$ is a line bundle on $C\times E$ fixed by the action of $G$ such that $\phi^*\O_{(C\times E)/G}{(\im(i_{C/G}))}.E=1$.

 Let $L$ be as in the hypothesis and $f=\alpha(L)$: because $L$ is fixed by the action of $G$, thanks to Proposition \ref{rmpic}, we can say that also $f$ is. In particular we obtain 
\begin{equation*}
f(g^{-1}c)-f(g^{-1}c_0)=f(c)-f(c_0)=f(c),
\end{equation*}
for every $g\in G$ or, equivalently, if we denote by $e_g=f(gc_0)$,
\begin{equation}
\label{eqG}
f(gc)-f(c)=e_g\in E.
\end{equation}

Let $\alpha_f\in \aut(C\times E)$ defined by  
\[(c,e)\mapsto(c,e-f(c));\]
this gives the following commutative diagram
\[
\begin{tikzcd}
\Pic(C\times E) \arrow{r}{\beta}\arrow{d}{(\alpha_f)_*} & \Hom_{c_0}(C,E) \arrow{r}{}\arrow{d}{h} & 0\\
\Pic(C\times E) \arrow{r}{\beta'} & \Hom_{c_0}(C,E) \arrow{r}{} & 0,
\end{tikzcd}
\]
where the map $h$ is bijective and defined by $h(\psi)=\psi-f$. Notice that $\beta'\circ(\alpha_f)_*(L)=h\circ\beta(L)=0$.    We would like to know how $G$ acts on $C\times E$ after this change of coordinates, i.e. what is $\alpha_f\circ g\circ \alpha^{-1}_f(c,e)$: we see that
\begin{equation}
\begin{split}
\alpha_f\circ g\circ \alpha_f^{-1}(c,e)=\alpha_f\circ g(c,e+&f(c))=\alpha_f(gc,e+f(c)+g)=\\
=(gc,e+f(c)+g-f(&gc))=(gc,e+g-e_g),
\end{split}
\label{eq_alfafinv}
\end{equation}
where the last equality follows by Equation \ref{eqG}. Notice that $(\alpha_f)_*(L).E=1$ and that $(\alpha_f)_*(L)$ is still $G$-invariant after conjugating the action of $G$ with $\alpha_f$, i.e. $(\alpha_f)_*g(\alpha_f)^*(\alpha_f)_*L=(\alpha_f)_*L$. 

If we assume that $g=e_g$ for all $g\in G$, then it is immediate by Equation \ref{eq_alfafinv}, that $(C\times E)/G=C/G\times E$ and we are done. Hence, suppose by contradiction that there exists $\gamma\in G$ such that $\gamma-e_\gamma\neq 0$. 
Because $\beta'((\alpha_f)_*L)=0$ and $(\alpha_f)_*L.E=1$, we have that 
\[(\alpha_f)_*L=C_e+\pi_C^*(B),\]
where $C_e$ is the fibre of the second projection over $e\in E$ and $B\in \Pic(C)$.  Because $(\alpha_f)_*L$ is $G$-invariant we can conclude that
\[(\alpha_f)_*L=(\alpha_f)_*\gamma_*(\alpha_f)^*(\alpha_f)_*L=C_{e+\gamma-e_\gamma}+\pi_C(\gamma_*(B'));\] 
in particular it follows $\gamma-e_\gamma=0$, a contradiction.
\end{proof}

\subsection{Double coverings}
\label{dblcov}
The material in this section is well known and for the results presented here we refer to \cite{barth}. Let $Y$ be a variety, $R$ be a reduced effective divisor (possibly $R=0$) and $L$ be a line bundle for which $L^2=\O_Y(R)$. It is then possible to define, provided that $Y$ is smooth, a ramified double covering $\pi\colon X\to Y$ (cf. \cite{barth} I.17) branched over $R$ satisfying the following properties:
\begin{itemize}
	\item $X$ is normal;
	\item let $R_1$ be the reduced divisor $\pi^{-1}(R)$, then $\pi^*(R)=2R_1$;
	\item $K_X=\pi^*(K_Y + L)$;
	\item $\pi_*(\O_X)=\O_Y\oplus L^{-1}$.
\end{itemize}
The singularities of $X$ are strictly related to the singularities of $R$; in particular if $R$ is smooth, then so is $X$. 

A classical way to solve singularities of a double cover of surfaces $\pi\colon X\to Y$ branched over $R$ is the canonical resolution (cf. \cite{barth} III.7):
\[
\begin{tikzcd}
\widetilde{X}=X_t \arrow{d}{\pi_t}\arrow{r}{\phi_t} & X_{t-1}\arrow{d}{\pi_{t-1}}\arrow{r}{\phi_{t-1}} & \dots \arrow{r}{\phi_2} & X_1\arrow{d}{\pi_1} \arrow{r}{\phi_1} & X \arrow{d}{\pi}\\
Y_t \arrow{r}{\psi_t} & Y_{t-1}\arrow{r}{\psi_{t-1}} & \dots \arrow{r}{\psi_2} & Y_1 \arrow{r}{\psi_1} & Y,
\end{tikzcd}
\]
where the $\psi_i$ are successive blow-ups that resolve the singularities of $R$, the morphism $\pi_i$ is the double cover branched over $R_i=\psi_i^*R_{i-1}-2m_iE_i$, where $E_i$ is the exceptional divisor of $\psi_i$, $m_i=\intinf{d_i/2}$ with $d_i$ the multiplicity in $R_{i-1}$ of the blown-up point and $\intinf{d_i/2}$ denotes the integral part of $d_i/2$. One has the following relations (cf. \cite{barth} V.22):
\begin{equation}
K_{\widetilde{X}}^2=2K_Y^2+2K_Y.R+\frac{1}{2}R^2-2\sum_{i=0}^{t-1}(m_i-1)^2,
\label{cansq}
\end{equation}
and
\begin{equation}
\chi(\O_{\widetilde{X}})=2\chi(\O_Y)+\frac{1}{4}K_Y.R+\frac{1}{8}R^2-\frac{1}{2}\sum_{i=0}^{t-1}m_i(m_i-1).
\label{eulchar}
\end{equation}
Recall that the singularities of the branch locus $R$ are said to be negligible if $m_i=1$ (or, equivalently, $d_i=2,3$) for all $i=0,\dots,t-1$: in this case $X$ is the  canonical model of $\widetilde{X}$ (cf. \cite{barth} III.7 table 1) and $K_{\widetilde{X}}=(\pi_t\circ\psi_t\circ\dots\circ\psi_1)^*(K_Y+L)$ (ibidem Theorem III.7.2).  
Moreover, if $Y$ contains no rational curves, we have that $\widetilde{X}$ is minimal (in general it can have exceptional divisors even if $Y$ contains no rational curves, cf. \cite{barth} III.7 table 1)

\begin{remark}
\label{alb}
Suppose that we have a double cover $\pi\colon X\to Y$ with non-trivial smooth branch divisor $R$. 
Then, if $q(X)=q(Y)$, it follows that $\alb_{\pi}\colon \Alb(X)\to\Alb(Y)$ is an isomorphism. 
Indeed, because $q(X)=q(Y)$, the morphism $\alb_{\pi}$ is an isogeny and so is, by duality, $\pi^*\colon\Pic^0 Y\to\Pic^0X$. 
Suppose that there exists a non-trivial element $\eta\in \ker({\pi}^*)$. 
This in particular means that $\eta$ is a torsion element ($\ker({\pi}^*)$ is a finite group) and ${\pi}^*(\eta)=0$. If we consider the \'etale cover $Z\to{Y}$ given by $\eta$ and we complete the diagram as follows
\[
\begin{tikzcd}
\bigsqcup_{i=1}^{ord(\eta)}{X} \arrow[dr, phantom, "\square"] \arrow[r] \arrow[d]  & {X}\arrow[d, "{\pi}"]\\ 
Z \arrow[r] & {Y},
\end{tikzcd}
\]
we see that ${\pi}$ factors through $Z$, but this is impossible because it has degree two and has ramification. So $\alb_{\pi}$ is an isomorphism.

We would like to stress that if, in the same hypothesis, we suppose that $\pi$ has no ramification, then $\alb_{\pi}$ is an isogeny of order two. Indeed any \'etale cover of degree two is induced by a torsion element $L\in\Pic^0(Y)[2]$ for which $\pi^*L=\O_X$.
\end{remark}

\section{Examples}

In this section we give explicit examples of surfaces which satisfy equalities in Theorems \ref{mioa} and \ref{miob}, proving that all the inequalities are sharp. First we give an example of a surface satisfying  equality in Theorem \ref{mioa} for $q\geq 3$ (a characterization of the surfaces satisfying  equality for $q=2$ is done in \cite{barparsto}).

\begin{example}[double cover of a product elliptic surface]
\label{es1}
We consider an elliptic surface $Y_0=C\times E$ which is the product of an elliptic curve $E$ and a curve $C$ of genus $g>1$.  With an abuse of notation, we call $E$ the class of a fibre of $\pi_C$ in $\ns(Y_0)$ and $C$ the class of a fibre of $\pi_E$ in $\ns(Y_0)$. 

\[
\begin{tikzcd}
	 & [-2em] C_e\arrow[mapsto]{r}{} \arrow[phantom,"\rotatebox{270}{$\subseteq$}"]{d}{} & e \arrow[phantom,"\rotatebox{270}{$\in$}"]{d}{}\\ [-3ex]
	E_c \arrow[phantom,"\subseteq"]{r}{} \arrow[mapsto]{d}{} & C\times E\arrow{r}{\pi_E} \arrow{d}{\pi_C} & E\\
	c \arrow[phantom,"\in"]{r}{} & C & 
\end{tikzcd}
\]
We know that every divisor of even degree on a curve is two-divisible in the Picard group: by this, it follows that $R_0\sim_{hom}2C+2dE$ ($d> 0$) is two-divisible in $\Pic(Y_0)$, i.e. there exists a line bundle $L_0$ such that $R_0=2L_0$, moreover $R_0.E=2$. There certainly exist elements in this homological class that are reduced and have at most negligible singularities: for example it is enough to take different fibres $C_1, C_2$ and $E_1,\dots,  E_{2d}$, then $C_1+C_2+E_1+\dots+E_{2d}$ has only double points (Actually, if $d\gg 0$, a general element of the homological class $R_0$ is smooth by Bertini). It is immediate that $K_{Y_0}.R_0=4(q-2)$, where $q=q(Y_0)=g+1$.

Thus we obtain a double cover $\pi_0\colon X_0\to Y_0$ and after the canonical resolution (cf. section \ref{dblcov}) we get a smooth surface $X$ and the following diagram:
\[
\begin{tikzcd}
X\arrow{r}{\phi}\arrow{d}{\pi} & X_0\arrow{d}{\pi_0}\\
Y\arrow{r}{\psi} & Y_0.
\end{tikzcd}
\] 

We know that, if the singularities are at most negligible, $K_X=(\phi\circ\pi_0)^*(K_{Y_0}+L_0)\sim_{\textit{Num}}(\phi\circ\pi_0)^*(C+(2q-4+d)E)$. It is easy to see, by the Nakai-Moishezon criterion, that $C+(2q-4+d)E$ is always ample and, from this, it follows that $X$ is of general type. Furthermore $X$ is minimal, because $Y_0$ is, and its canonical model is $X_0$.

By Equations \ref{cansq} and \ref{eulchar} we get
\[K_X^2=2(K_{Y_0}+\frac{1}{2}R_0)^2=8(q-2)+4d\]
and
\[\chi(\O_X)=q-2+d.\]
Moreover, because $L_0$ is ample and $(\pi_0)_*\O_{X_0}=\O_{Y_0}\oplus L_0^{-1}$, we have that $q(X)=q$.

In particular 
\[K_X^2-4\chi(\O_X)=4(q-2).\]
Remark \ref{alb} ensures that the Albanese varieties coincide and that the Albanese morphism of $X$ is composed with an involution.

Notice also that 
\[K_X^2-\frac{9}{2}\chi(\O_X)=4(q-2)-\frac{1}{2}(q-2)-\frac{1}{2}d=\frac{1}{2}(7(q-2)-d),\]
which is smaller than zero if and only if $d>7(q-2)$.

Hence we have proved that these surfaces satisfy the conditions of Theorem \ref{mioa}: the next step is to prove that they are the only ones; this will be done in Section 5.
\end{example}

Now we give three examples of surfaces for which equality holds in Theorem \ref{miob}: in the first two cases we have $q(X)\geq 3$, while in the last example $q(X)=2$.

\begin{example}
\label{buel}
The easiest possible case is a simple modification of Example \ref{es1}. We take $Y_0$ as before: in this case we just need to take $R_0\sim_{hom}4C+2dE$ and everything is verified in a completely similar way (as before, we need $d> 7(q-2)$). In this case we have $q(X)=q, K_X^2=16(q-2)+8d$ and $\chi(\O_X)=2(q-2)+2d$. Hence $K_X^2-4\chi(\O_X)=8(q-2)$.
\label{es3}
\end{example}

\medskip

Before the next example, we recall some facts that will be useful. It is known that the Jacobian variety $\jac C'$ of a general curve $C'$ of genus $g'$ is simple (cf. \cite{lange} Theorem 17.5.1). It is also known that, given a general \'etale double cover $C\to C'$ its Prym variety $P(C,C')$ is simple (cf. \cite{biswas} or \cite{pirola} Proposition 3.4). Because $P(C,C')$ is complementary to $\jac C'$ inside $\jac C$ (cf. \cite{lange} Section 12.4), by Poincar\'e's reducibility Theorem (ibidem Theorem 5.3.5), $\jac C$ is isogenous to $\jac C'\times P(C,C')$. In particular there are no Abelian subvarieties of codimension $1$ of $\jac C$ if $g'>2$ (if $g'=2$, the dimension of $P(C,C')$ is $1$ , in particular $\jac C'$ is an Abelian subvariety of codimension $1$ of $\jac C$). So, for every elliptic curve $E$, the set $\Hom(C,E)$ contains only constant morphisms.

\begin{example}[Double cover of a non-trivial smooth elliptic surface]
\label{es4}
Here we present  an example of surface $X$ of general type satisfying equality in Theorem \ref{miob}, whose canonical model is a ramified double cover of an elliptic surface which is not a product. We start with $C'$, $C$ and $E$ as above.

Let $G$ be a subgroup of order $2$ of $E$ acting freely on $C$ such that the quotient $C/G$ is $C'$: this action clearly extends diagonally to the product giving a finite morphism of degree two $f\colon C\times E\to Y_0:=(C\times E)/G$. Proposition \ref{Teo}, together with the non-existence of surjective morphisms from $C$ to $E$, show that there is no line bundle $L$ fixed by $G$ for which $L.E=1$. By Lemma \ref{lem2} this is enough to prove that $Y_0$ is not a product. We denote by $\pi_{C/G}$ and $\pi_{E/G}$ the two morphisms from $Y_0$ to $C/G$ and $E/G$ respectively, whose generic fibres are $E$ and $C$ respectively. We have the following commutative diagram:
\[
\begin{tikzcd}
E\arrow{r}{f_E} & E/G\\
C\times E\arrow[swap]{u}{\pi_E}\arrow{d}{\pi_C} \arrow{r}{f} & Y_0\arrow[swap]{u}{\pi_{E/G}}\arrow{d}{\pi_{C/G}}\\
C\arrow{r}{f_C} & C/G.
\end{tikzcd}
\]

Recall that, in our case, a line bundle on $C\times E$ descends to a line bundle on $Y_0$ if and only if its class in the Picard group is fixed by $G$  (cf. \cite{drezet} Theorem 2.3). Indeed, when the group $G$ is cyclic, it is always possible to give to a line bundle $L$ fixed by the action of $G$, the structure of a $G$-bundle.

Let $L_C$ and $R_C$ be two line bundles on $C/G$ of degree respectively $d$ and $2d$  such that $2L_C=R_C$. Similarly, let $L_E$ and $R_E$ be two line bundles on $E/G$ of degree respectively $1$ and $2$ such that $2L_E=R_E$. If we take an element $R\in |\pi_{C/G}^*R_C+\pi_{E/G}^*R_E|$ with at most negligible singularities (as in Example \ref{es1}, we can even assume that $R$ is smooth by Bertini if $d\gg 0$) and we denote by $L=L_C\boxtimes L_E$, we have $2L=\O_{Y_0}(R)$; hence we get a double cover $\pi_0\colon X_0\to Y_0$, and after the canonical resolution (cf. section \ref{dblcov}) we get a smooth surface $X$ and the following diagram:
\[
\begin{tikzcd}
X\arrow{r}{\phi}\arrow{d}{\pi} & X_0\arrow{d}{\pi_0}\\
Y\arrow{r}{\psi} & Y_0,
\end{tikzcd}
\]  
where $X$ is smooth and, because $L$ is ample, $q:=q(X)=q(Y_0)$. By Equations \ref{cansq} and \ref{eulchar}
\[K_X^2=2(K_{Y_0}+\frac{1}{2}R)^2=16(q-2)+8d\]
 and 
\[\chi(\O_X)=\frac{1}{4}(K_{Y_0}+\frac{1}{2}R).R=2(q-2)+2d.\]
In particular we have 
\[K^2_X-4\chi(\O_X)=8(q-2).\]
It is obvious (cf. Remark \ref{alb}) that the Albanese morphism of $X$ factors through $Y$ and, as in the previous examples, we require $d>7(q-2)$ in order to have that $K_X^2<\frac{9}{2}\chi(\O_X)$: this concludes our example.  Notice that the condition $g(C')> 2$ implies that $q(X)>3$: we do not know if there exists an example of $X$ for which the equality holds where the quotient by the involution is not a product when $q=3$.
Indeed if $g'=2$ (equivalently, $q=3$) we have that $\Hom_{c-0}(C,E)$ is not trivial.
\end{example}

\begin{example}[Double cover of an Abelian variety ramified over a divisor with a quadruple point]
This is an example of surface of general type with maximal Albanese dimension whose Albanese morphism is composed with an involution with $q(X)=2$ satisfying equality in Theorem \ref{miob}, i.e.
\[K^2_X=4\chi(\O_X)+2.\]

Let $A$ be an Abelian surface and let $H$ be a very ample divisor. Take general elements $H_1, H_2, H_3, H_4, H_5, H_6, H_7, H_8$  inside the linear system  $|H|$ such that they are smooth, they all pass through a point $p$ and $H_i$ intersects $H_j$ for $i\neq j$ transversely at every intersection point. 
Denote by $H_{i,j}=(H_i\cap H_j)\setminus \{p\}$ for $1\leq i<j \leq 8$: we also require that $H_{i,j}\cap H_{i',j'}=\emptyset$ for every $(i,j)\neq(i',j')$. 
Let $D_1=H_1+H_2+H_3+H_4$ and $D_2=H_5+H_6+H_7+H_8$ and consider the pencil $P=\lambda D_1+\mu D_2$: the base locus of $P$ is $\{p\}\sqcup \bigsqcup_{1\leq i,j\leq 4} H_{i,j+4}$. By a Bertini-type argument, the generic element $R\in P$ is smooth away from $p$ and has a quadruple ordinary point at $p$. 

It is obvious that $R$ is two-divisible, i.e. there exists $L\in\Pic(A)$ with $2L=\O_A(R)$. Consider the double cover $\pi\colon X\to A$ branched over $R$, and take the canonical resolution $\widetilde{\pi}\colon\widetilde{X}\to\widetilde{A}$ (cf. section \ref{dblcov}). Because $R$ has a single quadruple ordinary point, $\widetilde{X}$ is the minimal smooth model of $X$ (cf. \cite{barth} III.7). Denote by $E$ the exceptional divisor of $\psi\colon\widetilde{A}\to A$: we have that $\widetilde{L}^2=(\psi^*L-2E)^2=L^2-4>0$  if we assume $L$ to be sufficiently ample.
Moreover, up to take a multiple of $H$, we may suppose that the Seshadri constant of $L$ is sufficiently big such that $\widetilde{L}$ intersects positively every curve on $\widetilde{A}$ (cf. \cite{lazarsfeldi} Definition 5.1.1 and Example 5.1.4). 
In particular $\widetilde{L}$ is very ample by the Nakai-Moishezon criterion, from which follows $q(\widetilde{X})=q(A)=2$ and $X$ is of general type. By Equations \ref{cansq} and \ref{eulchar}, we have that
\[K_{\widetilde{X}}^2= \frac{1}{2}R^2-2\]
and
\[\chi(\O_{\widetilde{X}})=\frac{1}{8}R^2-1.\]
In particular we have
\[K_{\widetilde{X}}^2-4\chi(\O_{\widetilde{X}})=2\]
and clearly (cf. Remark \ref{alb}) the Albanese morphism of $X$ is the natural morphism to $A$, moreover if one takes the ample divisor $H$ such that $H^2>2$, one easily derives $K_X^2<\frac{9}{2}\chi(\O_X)$.

\end{example}

\section{Severi Type inequalities}
In this section we briefly recall the main ideas of Lu and Zuo in their paper \cite{lu} that will be used in our proofs.  As stated there (Theorem 3.1) the condition
\[K_X^2<\frac{9}{2}\chi(\O_X)\] 
is necessary to prove that there exists an involution $i\colon X\to X$ with respect to which the Albanese morphism $\alb_X$ is stable (or, equivalently, $\alb_X$ is composed with $i$), i.e.
\[
\begin{tikzcd}
X \arrow{rr}{i}\arrow[swap]{dr}{\alb_X}& & X\arrow{dl}{\alb_X} \\
 & \Alb(X). &
\end{tikzcd}
\]

\begin{remark}
\label{rem_axb}
Notice that the condition "$\alb_X$ is composed with an involution" is necessary. Otherwise it is easy to construct a counter example. Take a product of curves $A$ and $B$ with $g(A)=2$ and $g(B)\geq 2$. Then the surface $X=A\times B$, which is of general type, gives the desired counterexample. Indeed, the Albanese morphism is clearly injective and, by K\"unneth formula and the formula of the canonical bundle of a product, we have that 
\[K_X^2=8(g(B)-1) \quad \chi(\O_X)=g(B)-1.\]
Hence 
\[K_X^2-4\chi(\O_X)=4(g(B)-1)=4(q(X)-3)<4(q(X)-2).\]

Notice also that, at least for Theorem \ref{miob}, it is also necessary the condition $K_X^2<\frac{9}{2}\chi(\O_X)$ even if we are assuming that the Albanese morphism is composed with an involution.
Let $Y=B\times B$, where $B$ is as above and denote by $B_i$ a fibre of the i-th projection $\pi_i\colon Y\to B$.
The invariants of $Y$ are:
\begin{itemize}
	\item $q(Y)=4$;
	\item $p_g(Y)=4$,
	\item $\chi(\O_Y)=1$;
	\item $K_Y\sim_{num}2B_1+2B_2$;
	\item $K_Y^2=8$.
\end{itemize}
Let $L=B_1+B_2$ and take $R$ a general element of $|2L|$ (in particular $R\sim_{num} K_Y$) which, by Bertini, may be assumed to have at most negligible singularities.
Then the desingularization $X$ of the double cover defined by $L^2=\O_Y(R)$ satisfies (cf. Remark \ref{alb} and Equations \ref{cansq} and \ref{eulchar})
\begin{itemize}
	\item $q(X)=q(Y)$ and hence they have the same Albanese variety;
	\item $\chi(\O_X)=5$;
	\item $K_X^2=36$,
\end{itemize}
from which we obtain
\[K_X^2-4\chi(\O_X)=36-20=16=8(q-2)\geq \frac{5}{2}=\frac{1}{2}\chi(\O_X).\]
\end{remark}

\medskip

The quotient surface $Y=X/i$ can be  singular, but its singular points are not so bad: they are $A_1$ singularities and they are in one-to-one correspondence with the isolated fixed points of $i$. Let $Y'$ be the resolution obtained by blowing up the singularities and let $X'$ be the blow-up of $X$ over the isolated fixed points of $i$. Denote by $Y_0$ the minimal model of $Y'$ and by $X_0$ the middle term of the Stein factorization of the morphism from $X'$ to $Y_0$. What we get is the following commutative diagram.
\[
\begin{tikzcd}
X \arrow{d}{\pi} & X' \arrow{d}{\pi'}\arrow{l}[swap]{f_X}\arrow{r}{g_X} & X_0 \arrow{d}{\pi_0}\\
Y=X/i 					 & Y' \arrow{r}{g_Y} \arrow{l}[swap]{f_Y}								& Y_0.
\end{tikzcd}
\]
 We know that the double covers $\pi'$ and $\pi_0$ are  given by equations $2L'=\O_{Y'}(R')$ and $2L_0=\O_{Y_0}(R_0)$ respectively where $R'$ and $R_0$ are the branch divisors. Notice that $R_0$ has to be reduced (because $X_0$ is normal), while $R'$ has to be smooth (because $X'$ is smooth). It follows directly from the universal property of the Albanese morphism and the fact that $\alb_X$ factors through $\pi$ that $Y_0$ is a surface of maximal Albanese dimension with $q(Y_0)=q$.

By the classification of minimal surfaces, we know that $Y_0$ has non-negative Kodaira dimension and maximal Albanese dimension and in particular we have the following possibilities :
\begin{itemize}
	\item if $k(Y_0)=0$, then $Y_0$ is an Abelian surface and $q=2$;
	\item if $k(Y_0)=1$, then $Y_0$ is an isotrivial smooth elliptic surface over a curve $C$ with genus $g(C)=q-1$ and $q\geq 3$;
	\item if $k(Y_0)=2$, then $Y_0$ is a minimal surface of general  type of maximal Albanese dimension with $q\geq 2$.
\end{itemize}

First, we restrict to the case $k(Y_0)<2$. The surface $X_0$ may not be smooth, so we perform the canonical resolution (cf. section \ref{dblcov}). We get the following diagram
\[
\begin{tikzcd}
X \arrow{d}{\pi} & X_t \arrow{d}{\pi_t}\arrow{l}[swap]{\phi}\arrow{r}{\phi_0} & X_0 \arrow{d}{\pi_0}\\
Y=X/i & Y_t \arrow{r}{\psi_0} \arrow{l}[swap]{\psi} & Y_0.
\end{tikzcd}
\]
 We notice that $X$ is nothing but the minimal model of $X_t$: thus, there exists an integer $n$ such that $\phi$ is the composition of $n$ blow-ups. In particular $K_X^2=K_{X_t}^2+n$ and $\chi(\O_X)=\chi(\O_{X_t})$.

If $Y_0$ is an elliptic surface over a curve $C$ with $g(C)=q-1$, then, denoting by $F$ a general elliptic fibre of the fibration, $F.R_0=2F.L_0\geq 2$. Indeed if $F.R_0=0$, then we would have that $X$ has an elliptic fibration, which is not the case because $X$ is of general type. Recall that the numerical class of the canonical bundle of such an elliptic surface is (see \cite{barth} V.12.3)
\[K_{Y_0}\sim_{\textit{Num}} (2g(C)-2)F+\sum_{j=1}^m(n_j-1)F_j=2(q-2)F+\sum_{j=1}^m(n_j-1)F_j,\]
where $n_jF_j$ are the the multiple fibres with $F_j$ reduced.
When $Y_0$ is Abelian, we know that the canonical bundle is trivial.

Summarizing we get (thanks to Equations \ref{cansq} and \ref{eulchar})
\begin{equation}
\label{eqell1}
\begin{split}
K_X^2-4\chi(\O_X)=K_{X_t}^2&-4\chi(\O_{X_t})+n=\\
=2(K_{Y_0}^2-4\chi(\O_{Y_0}))+K_{Y_0}&.R_0+2\sum_{i=0}^{t-1}(m_i-1)+n=\\
=2(q-2)F.R_0+\sum_{j=1}^m(n_j-1&)F_j.R_0+2\sum_{i=0}^{t-1}(m_i-1)+n\geq\\
\geq 4(q-2)+\sum_{j=1}^m(n_j-1)F_j.R_0&+2\sum_{i=0}^{t-1}(m_i-1)+n\geq 4(q-2).
\end{split}
\end{equation}

So equality holds if and only if $n=0$ (i.e. $X=X_t$), $m_i=1$ for all $i$, $n_j=1$ for all $j$ (this is required only when $Y_0$ is elliptic) and $K_{Y_0}.R_0=4(q-2)$. The last condition is trivially true in the case $Y_0$ is an Abelian surface, while in the case of an elliptic surface it tells us that there are no multiple fibres  and from this it follows that, after a suitable base change, $Y_0$ is a product of an elliptic curve with a curve of higher genus (cf. \cite{serrano}). The condition $m_i=1$ implies that the singularities of the branch divisor $R_0$ are at most negligible. Hence the inverse image of  an exceptional curve is a union of $-2$-curves (cf. \cite{barth} III.7 Table 1). This, together with the fact that $Y_0$ has no rational curves, implies that $X_0$ is the canonical model of $X$.

\begin{remark}
\label{uguaglianza}
We stress here what are the necessary numerical conditions on $Y_0$ (in the case it is an elliptic surface) in order to satisfy the equality of Theorem \ref{mioa}. Looking at Equation \ref{eqell1} it is immediate that this happens if and only if 
\begin{itemize}
	\item $F.R_0=2$;
	\item $n_j=1\quad \forall\ j$;
	\item $m_i=1\quad \forall\ i$;
	\item $n=0$.
\end{itemize}
The same conditions have to be verified in the case $Y_0$ is Abelian without the condition on the multiple fibres.
\end{remark}

Consider the case when $Y_0$ is of general type. By Theorem 1.3 of \cite{lu}, there are two possibilities. First we assume that $K_{Y_0}^2\geq 4\chi(\O_{Y_0})+4(q-2)$. As before we obtain
\begin{equation}
\label{eqgen1}
\begin{split}
K_X^2-4\chi(\O_X)&\geq 2(K_{Y_0}^2-4\chi(\O_{Y_0}))+K_{Y_0}.R_0+2\sum_{i=0}^{t-1}(m_i-1)\geq\\
&\geq 8(q-2)\geq 4(q-2).
\end{split}
\end{equation}
Equality would be possible if $q=2$, but it is shown that this is not the case (cf. \cite{lu} proof of Theorem 1.3 or \cite{barparsto} Theorem 1.1).

The other possible case is when $K_{Y_0}^2\geq \frac{9}{2}\chi(\O_{Y_0})$. In this case we have that 
\begin{equation}
\label{eqgen2}
K_X^2-4\chi(\O_X)>64 \max\{q-3,1\}
\end{equation}
(\cite{lu} Lemma 4.4), i.e. we have a much stronger inequality (it is in this step that the condition $K_X^2<\frac{9}{2}\chi(\O_X)$ is really needed).

\section{Proof of Theorem \ref{mioa}}

In this section we are going to prove Theorem \ref{mioa}. This will be done in two steps: first, we show that all the possible examples of a $2$-divisible divisor $R$ in an elliptic surface $Y=C\times E$ which intersects the elliptic fibre twice are linearly equivalent to those  in Example  \ref{es1}. The main tools of this first part are the See-saw Principle (Theorem \ref{saw}) and the explicit  formula for the Picard group of $C\times E$ (Proposition \ref{Teo}). Then we will see that the equalities in Theorems \ref{mioa} and \ref{miob} are stable under \'etale base change coming from the base of the elliptic fibration which, thanks to Lemma \ref{lem2}, will imply that $Y$ has to be a product.

Let, as usual, $C$ be a curve of genus $g(C)>1$ and $E$ be an elliptic curve. We have the following Lemmas.

\begin{lemma}
Let $Y=C\times E$, then a double cover of $Y$ branched over a divisor $R$ linearly equivalent to
\[\Gamma_f+\Gamma_{f+e}+\sum_{i=1}^{2d}E_i,\]
where $f\in\Hom(C,E)$, is isomorphic to a double cover of $C\times E$ branched over a divisor $R'$ linearly equivalent to
\[C_1+C_2+\sum_{i=1}^{2d}E_i.\]
\label{es2}
\end{lemma}
\begin{proof}
In order to prove this, it is enough to see that there exists an automorphism of $C\times E$ which sends $R$ to $R'$. We will prove even more: actually the elliptic fibres of $\pi_C$ will be fixed by this automorphism. Indeed let $\pi_f\colon C\times E\to E$ be the morphism given by $(c,e)\mapsto e-f(c)$, we notice that $\pi_f^{-1}(e)=\Gamma_{f+e}$. In particular all the graphs $\Gamma_{f+e}$ are equivalent in $\ns(C\times E)$. Notice that if we consider the automorphism $\alpha_f$ of $C\times E$ defined by $(c,e)\mapsto (c,e-f(c))$, this clearly fixes the fibres of $\pi_C$ with respect to which $\pi_f$ is the second projection, i.e. we have the following commutative diagram

\[
\begin{tikzcd}
 & E \\
C\times E \arrow{d}{\pi_C}\arrow[ur, "\pi_f" ] \arrow{r}{\alpha_f}  & C\times E\arrow{d}{\pi_C}\arrow[swap]{u}{\pi_E}\\
C\arrow{r}{Id} &  C:
\end{tikzcd}
\]
this concludes the proof. 
\end{proof}

\begin{remark}
Lemma \ref{es2} says that the condition on the branch divisor $R$ in Theorem \ref{mioa} can be given in a seemingly weaker way i.e. we can replace the fibres $C_1$ and $C_2$ by the translated graphs $\Gamma_f$ and $\Gamma_{f+e}$.
\end{remark}

\begin{lemma}
\label{lem1}
Let $Y=C\times E$, $R$ be an effective reduced divisor such that $R.E=2$ and there exists a line bundle $L$ satisfying $2L=R$. Then 
\[R\sim_{lin}\Gamma_{g+e_1}+\Gamma_{g+e_2}+\sum_{i}E_i,\]
where $g\in \Hom_{c_0}(C,E)$ and $e_i$ are elements in $E$.
\end{lemma}

\begin{proof}
Let $R$ be as in the hypothesis and let $e_1, e_2\in E$ be the two points such that $(c_0,e_1), (c_0,e_2)\in R\cap E_{c_0}$ (it could be that $e_1=e_2$). Suppose that $\beta(R)=f$ (cf. Proposition \ref{Teo}): because $R=2L$ and the map $\beta$ is a group morphism, it follows that there exists an element $g\in \Hom_{c_0}(C,E)$ such that $2g=f$. 

Consider now the divisor 
\[D=R-\Gamma_{g+e_1}-\Gamma_{g+e_2}.\]
It is then obvious that $\beta(D)=0$ and, moreover, $D$ restricted to each fibre $E_c$ is trivial. Indeed we have that
\[0=\beta(D)(c)=i_c^*(D)-i_{c_0}^*(D)=i^*_c(D).\]
Using the See-saw Principle (cf. Theorem \ref{saw}) on $D$ we see that
\[D\sim_{lin}\pi_C^*(B)\]
i.e. 
\[R\sim_{lin}\Gamma_{g+e_1}+\Gamma_{g+e_2}+\pi^*_C(B).\]

It is possible to  show that $B$ has positive degree. Indeed, applying the isomorphism $\alpha_g$ (cf. Lemma \ref{es2}), we may assume that 
\[R\sim_{lin}\pi_E^*(A)+\pi_C^*(B)\]
where $A$ is a degree two divisor of $E$. Then, by K\"unneth formula, we have
\[h^0(R)=h^0(A)h^0(B)\]
which is positive if and only if the degree of $B$ is positive. Summing up, we have 
\[R\sim_{lin}\Gamma_{g+e_1}+\Gamma_{g+e_2}+\sum_iE_i.\qedhere\]
\end{proof}

We are now ready to prove a proposition that tells us that equalities of the type
\begin{equation}
\alpha K_X^2+\beta \chi(\O_X)+\gamma (q-2)=0
\label{eq:a}
\end{equation}
behave well with respect to \'etale covers coming from the base of the elliptic fibration.

Let $Y$ be an elliptic surface of maximal Albanese dimension over a curve $C$ of genus $g(C)=q(Y)-1$ and let $X$ be the minimal smooth model of the double cover given by the equation $2L=\O_Y(R)$ (where $R$ is supposed to be ample and with at most negligible singularities) and denote by $\pi\colon X\to Y$ the induced morphism.

\begin{lemma}
\label{etale}
 In the above settings, $q(X)-2$ is multiplicative with respect to \'etale covers coming from  $C$. This means that if we consider an \'etale cover  $\gamma \colon \widetilde{C}\to C$ of degree $d$ and we take the base change
\[
\begin{tikzcd}
\widetilde{X}\arrow[r,"\beta"]\arrow[d,"\widetilde{\pi}"] & X\arrow[d,"\pi"]\\
\widetilde{Y}\arrow[r,"\delta"]\arrow[d,"\widetilde{\alpha}"] & Y\arrow[d,"\alpha"]\\
\widetilde{C}\arrow[r,"\gamma"] & C,
\end{tikzcd}
\]
 then 
\[q(\widetilde{X})-2=d(q(X)-2).\]
\end{lemma}

\begin{proof}
We know that $q(\widetilde{X})=g(\widetilde{C})+1$ and $q(X)=g(C)+1$ because $Y$ and $\widetilde{Y}$ have maximal Albanese dimension.  
Hence, applying Riemann-Hurwitz formula on $\gamma$, we get
\[q(\widetilde{X})-2=g(\widetilde{C})-1=d(g(C)-1)=d(q(X)-2).\qedhere\]
\end{proof}

\begin{remark}
\label{rem_bis}
The importance of the Lemma \ref{etale} in our discussion is given by the following result on elliptic surfaces. It is  known (cf. \cite{serrano}) that given an isotrivial smooth elliptic surface $Y$ on $C$ with  fibres isomorphic to $E$ there exists a suitable Galois \'etale base change giving the following Cartesian diagram:
\[
\begin{tikzcd}
\widetilde{Y}=\widetilde{C}\times E \arrow{r}{\delta}\arrow{d}{\widetilde{\alpha}} \arrow[dr, phantom, "\square"] & Y\arrow{d}{\alpha}\\
\widetilde{C} \arrow{r}{\gamma} & C,
\end{tikzcd}
\]
where the horizontal arrows are \'etale morphisms of degree $d$. In particular there exists a group $G$ acting freely on $E$ and $\widetilde{C}$ such that
\begin{itemize}
	\item $\widetilde{C}/G=C$;
	\item $\widetilde{Y}/G=Y$,
\end{itemize}
where the action on the product is the diagonal one and the quotient maps are given by $\gamma$ and $\delta$. Moreover the pullback of the  elliptic fibre of $\alpha$ is numerically equivalent to $d$ elliptic fibres of $\widetilde{\alpha}$.

In view of this, Lemma \ref{etale} tells us that every surface $X$ satisfying equality in Theorem \ref{miob} with irregularity $q\geq 3$ is an \'etale quotient of a surface $\widetilde{X}$ satisfying the same equality whose minimal model is a double cover of a product elliptic surface $\widetilde{Y}=\widetilde{C}\times E$ branched over a divisor $\widetilde{R}$ for which $\widetilde{R}.E=2$.
\end{remark}

\begin{proof}[Proof of Theorem \ref{mioa}]
This is a direct consequence of Lemma \ref{lem1}, Lemma \ref{lem2} and Remark \ref{rem_bis}. By \cite{lu} Theorem 1.3, we know that the canonical model of $X$ is isomorphic to a double cover of an isotrivial smooth elliptic surface $Y$, with fibre isomorphic to $E$. Moreover this covering is branched over a divisor $R$, with at most negligible singularities, for which $R.E=2$. Combining Remark \ref{rem_bis} with Lemma \ref{lem2} we prove that $Y$ is a product and Lemma \ref{lem1} gives the linear class of the branch divisor.
\end{proof}

\section{Proof of Theorem \ref{miob}}

In this section we  prove Theorem \ref{miob}.

\begin{proof}[Proof of Theorem \ref{miob}]
Recall, by section 4, that we have the following diagram:
\[
\begin{tikzcd}
X \arrow{d}{\pi} & X_t \arrow{d}{\pi_t}\arrow{l}[swap]{\phi_X}\arrow{r}{\phi_t} & X_{t-1}\arrow{d}{\pi_{t-1}}\arrow{r}{\phi_{t-1}} & \dots \arrow{r}{\phi_2} & X_1\arrow{d}{\pi_1} \arrow{r}{\phi_1} & X_0 \arrow{d}{\pi_0}\\
Y=X/i & Y_t \arrow{r}{\psi_t} \arrow{l}[swap]{\psi_Y} & Y_{t-1}\arrow{r}{\psi_{t-1}} & \dots \arrow{r}{\psi_2} & Y_1 \arrow{r}{\psi_1} & Y_0.
\end{tikzcd}
\]

If $Y_0$ is of general type,  the first part of the theorem is proven by equations \ref{eqgen1} and \ref{eqgen2}. In the case $Y_0$ is an Abelian surface the first part is trivial. 

So assume that $Y_0$ is an elliptic surface over a curve $C$ with maximal Albanese dimension. By the classification of surfaces we have that $K_{Y_0}^2=0=\chi(\O_{Y_0})$. The map $\phi_X$ is nothing but a sequence of $n$ blow-ups, in particular $K_X^2=K_{X_t}^2+n$. Moreover the numerical class of the canonical bundle of $Y_0$ is 
\[K_{Y_0}\sim_{\textit{Num}} (2g(C)-2)F+\sum_{j=1}^m(n_j-1)F_j=2(q-2)F+\sum_{j=1}^m(n_j-1)F_j,\]
where $F$ is a general fibre and $n_jF_j$ are the multiple fibres with  $F_j$ reduced.

Rephrasing Equation \ref{eqell1}, we obtain
\begin{equation}
\begin{split}
K_X^2-4\chi(\O_X)=K_{X_t}^2&-4\chi(\O_{X_t})+n=\\
=2(K_{Y_0}^2-4\chi(Y_0))+K_{Y_0}&.R_0+2\sum_{i=0}^{t-1}(m_i-1)+n=\\
=2(q-2)F.R_0+\sum_{j=1}^m(n_j-1&)F_j.R_0+2\sum_{i=0}^{t-1}(m_i-1)+n.
\end{split}
\label{eqell2}
\end{equation} 

We already know that $F.R_0$ is divisible by $2$ (recall that there exists $L_0$ such that $2L_0=R_0$) and strictly positive (otherwise $X$ would  be elliptic). As we have already noticed, the conditions in Theorem \ref{mioa}  are equivalent to the following:
\begin{itemize}
	\item $F.R_0=2$;
	\item $n_j=1\quad \forall\ j$;
	\item $m_i=1\quad \forall\ i$;
	\item $n=0$.
\end{itemize}

If we want to increase slightly $K_X^2-4\chi(\O_X)$, we thus have $4$ possibilities.

First we discuss $n$. We know that if all the $m_i=1$, then all the irreducible components of the exceptional curve in the covering surface are $(-2)$-curves (cf. \cite{barth} table 1 page 109). Moreover these are the only possible rational curves on $X_t$ (ibidem). This means that in this case $n=0$. In particular, if $n>0$, then there exists an $i$ such that $m_i>1$.

Now suppose that there exists an $i$ such that $m_i>1$. By the classification of simple singularities of curves (cf. \cite{barth} II.8) we know that we have two possibilities for $R_0$. If $R_0$ has a singular point $x$ of order greater or equal to $4$, then $(F.R_0)_x\geq 3$ (it may happen that one of the irreducible components of $R_0$ passing through $x$ is a fibre). Hence $F.R_0\geq 4$ because $R_0=2L_0$. The other possibility is that $R_0$ has a triple point $x$ which is not simple. A necessary condition for a  triple point not to be simple is to have a single tangent. If the tangent of $R_0$ in $x$ is transversal to $F$, then $(F.R_0)_x\geq 3$, conversely if it is tangent to $F$, we have $(F.R_0)_x\geq 4$ (it may happen, as before, that one of the irreducible component is $F$ itself). In both cases we have $F.R_0\geq 4$.

Suppose now that $Y_0$ has a multiple fibre $F_j$ with multiplicity $n_j$. In this case we have $F.R_0=n_jF_j.2L_0\geq 2n_j\geq 4$. 

To summarize, whatever quantity we increase, we get $F.R_0\geq 4$: that is to say that whenever 
\[K^2_X>4\chi(\O_X)+4(q-2),\]
we get
\[K^2_X\geq 4\chi(\O_X)+8(q-2)\]
and part 1 is proven.

Now we study the case when $q=2$. First we assume that $Y_0$ is an Abelian surface. In this case we have:
\begin{equation}
\begin{split}
K_X^2-4\chi(\O_X)=K_{X_t}^2&-4\chi(\O_{X_t})+n=\\
=2(K_{Y_0}^2-4\chi(Y_0))+K_{Y_0}&.R_0+2\sum_{i=0}^{t-1}(m_i-1)+n=\\
=2\sum_{i=0}^{t-1}(m_i&-1)+n.
\end{split}
\end{equation}
With the same argument as in part 1 of the proof, if $n>0$, then there exists an $i$ such that $m_i>1$. Then $K_X^2-4\chi(\O_X)> 0$ implies that there exists an $i$ such that $m_i>1$: in particular $K_X^2-4\chi(\O_X)\geq 2$.

Now suppose that $Y_0$ is of general type. If $K_{Y_0}\geq \frac{9}{2}\chi(\O_{Y_0})$, the proof is immediate thanks to Equation \ref{eqgen2}. In the case $K_{Y_0}^2-4\chi(\O_{Y_0})\geq 4(q-2)$ we have
\begin{equation}
\begin{split}
K_X^2-4\chi(\O_X)=2(K_{Y_0}^2-4\chi(\O_{Y_0}))+K_{Y_0}.R_0+2\sum_{i=0}^{t-1}(m_i-1)+n=\\
=2\Bigl(K_{Y_0}^2-4\chi(\O_{Y_0})+K_{Y_0}.L_0+\sum_{i=0}^{t-1}(m_i-1)\Bigr)+n:
\end{split}
\end{equation}
as above, if $n>0$, there exists an $i$ such that $m_i>1$. 
In particular $K_X^2>4\chi(\O_X)$ implies $K_X^2\geq 4\chi(\O_X)+2$ and part two of the Theorem is proven.

Now suppose that equality holds: it is enough to prove that $Y_0$ is an elliptic surface. Indeed, if it is, it is immediate from Equation \ref{eqell2} that the conditions of the Theorem are necessary and sufficient. 

Suppose by contradiction that $Y_0$ is a surface of general type. For the numerical invariants of $Y_0$ we have two possibilities. First, if $K_{Y_0}^2\geq \frac{9}{2}\chi(\O_X)$, we know that $K_X^2-4\chi(\O_{Y_0})\geq 64\max\{q-3,1\}$ (this is Equation \ref{eqgen2}): then $K^2_X>4\chi(\O_X)+8(q-2)$, a contradiction. The other possible case is if 
\[\frac{9}{2}\chi(\O_{Y_0})>K_{Y_0}^2\geq 4\chi(\O_{Y_0})+4(q-2).\]
If this happens, we have
\[K_X^2-4\chi(\O_X)=2(K_{Y_0}^2-4\chi(\O_{Y_0}))+K_{Y_0}.R_0+2\sum_{i=0}^{t-1}(m_i-1)+n.\]
By \cite{lu} we have  $K_{Y_0}^2-4\chi(\O_{Y_0})\geq 4(q-2)$: thus $K_X^2-4\chi(\O_X)=8(q-2)$ if and only if $K_{Y_0}^2-4\chi(\O_{Y_0})=4(q-2)$, $K_{Y_0}.R_0=0$, $X=X_t$ and $m_i=1$ for $i=0,\dots t-1$. In particular $K_{Y_0}.R_0=0$ implies that $R_0$ can only contain $(-2)$-curves because $Y_0$ is minimal of general type (cf. \cite{barth}  Theorem VII.5.1). If $R_0$ is not empty, then we have that 
\[(K_{Y_0}+\frac{1}{2}R_0).R_0=\frac{1}{2}R_0^2<0.\]
Since $K_X$ is equal to the pull-back of $K_{Y_0}+\frac{1}{2}R_0$, this equation tells us that $K_X$ is not nef, which is a  contradiction.
By this we have proved $X_0=X_t=X$, $Y_0=Y_t=Y$ and $\pi\colon X\to Y$ is an \'etale double cover such that $\alb_X=\alb_Y\circ\pi$ which is clearly a contradiction thanks to Remark \ref{alb}.
\end{proof}

\bibliographystyle{plain}
\bibliography{bibliografia}

\end{document}